\documentclass{aptpub}
\authornames{J. G. Ronan}
\shorttitle{Markov Kernels for OT}

\usepackage{dsfont}
\usepackage{amssymb}
\usepackage{amsmath}
\usepackage{mathrsfs}
\usepackage{color,xcolor}

\newcommand{\Pro}{\mathscr{P}}
\newcommand{\ACP}{\mathscr{ACP}}

\renewcommand{\supp}{\operatorname{supp}}
\newcommand{\dist}{\operatorname{dist}}
\newcommand{\proj}{\operatorname{proj}}

\newcommand{\g}{\gamma}

\newcommand{\pth}{$p^{th}$}

\usepackage[normalem]{ulem}

\begin{document}

\title{Markov Kernels in Optimal Transport via\\ Extending c-Cyclic Monotonicity}

\author[U.S. Army Corps of Engineers]{James G Ronan}
\address{Cold Regions Research and Engineering Laboratory, Hanover, NH 03755}

\begin{abstract}
In this paper we show that we can use Markov kernels as a model for optimal transport. 
This new framework can be easily translated into the standard coupling formulation of optimal transport, and we show that we can use a coupling as a Markov kernel for many optimal transport problems. 
Using kernels allows us to extend optimal transport to signed measures and treats the support of the measure as the salient feature.
This approach reveals additional structure for one-dimensional signed optimal transport. 
\end{abstract}

\keywords{Optimal Transport, Markov Kernels}
\ams{49Q22}{60G07}

\section{Introduction}\label{sec1}
The study of optimal transport has a long history dating from the 1781 work of Mongè in \cite{monge1781memoire} which established what corresponds to the transport map formulation of the subject.
Throughout optimal transport's subsequent development, the problem has been rephrased multiple times, each time revealing a new facet of the theory and spurring new developments and applications.
One particularly important example of this was Kantorovich's reformulation of the problem in 1942 \cite{kantorovich1942translocation} which cast the problem in terms of joint probability measures by introducing the notion of the optimal coupling.
This enabled optimal transport problems to be formulated as linear programming problems and made them computationally feasible. 

In the 1990's and early 2000's multiple authors began to explore geometric structure in optimal transport and its connection to fluid dynamics \cite{benamou2000computational, brenier1991polar, mccann1997convexity,  otto2001geometry}.
Entropic regularization and the Sinkhorn algorithm, \cite{cuturi2013sinkhorn}, led to even more computational tractability and spurred work in applying optimal transport to image analysis, \cite{parno2019remote,solomon2015convolutional}. 
Semi-discrete optimal transport limits the class of distributions considered and simplifies the dual-problem to optimal transport. 
These advantages fostered interest in new application domains and brought tools from computational geometry into the field \cite{aurenhammer1998minkowski,de2012blue, kitagawa2016convergence, levy2015numerical, levy2010p, levy2018notions}. 

Brief introductions to optimal transport can be found in \cite{ambrosio2013user, mccann2011five, santambrogio2014introduction}, while more thorough treatments are available in \cite{ambrosio2008gradient, peyre2019computational, santambrogio2015optimal, villani2003topics, villani2008optimal}.

Recently there has been increased interest in signed optimal transport, which applies techniques from optimal transport to signed measures. 
Signed optimal transport in one dimension has already proven itself useful in estimating parameters for seismic imaging, \cite{engquist2013application, engquist2016optimal, engquist2018seismic}, where multiple techniques were used to change the signed signal into a positive probability measure.
Those techniques included splitting into components, exponentiating, and adding a large constant.
However, such approaches for signed optimal transport have been application specific and consequently ad hoc and a more general theory for signed measures has yet to be established.

This paper establishes a new interpretation of optimal transport through the use of Markov kernels, a special case of transition kernels. 
The Markov kernel framework enables a new approach to signed optimal transport that is distinct from those described in \cite{ambrosio2011gradient, mainini2012description, piccoli2013transport, piccoli2014generalized, piccoli2016properties, piccoli2019wasserstein}.
We develop a path from standard optimal couplings to optimal transport kernels and show how the two interpretations are related. 
The interpretability between kernels and couplings enables current techniques that find optimal couplings to be re-purposed to make optimal transport kernels.

This perspective naturally leads to the treatment of signed optimal transport  by sending a signed measure through an optimal transport kernel. 
The connections between the kernel and optimal couplings ensure that the signed optimal couplings retain the geometric information that makes optimal transport useful. 
This use of optimal transport kernels recreates the capability of transport maps, while retaining the ease and computational advantages of searching for an optimal coupling due to the connections between kernels and couplings. 

The rest of this paper is organized as follows. In Section~\ref{sec2} we provide the necessary background on optimal transport.
Section~\ref{sec3} presents results extending optimal couplings towards the goal of introducing Markov kernels for optimal transport and
Section~\ref{sec4} describes how we can use Markov kernels for optimal transport, including their connection to optimal couplings.
Section~\ref{sec5} applies optimal transport kernels to signed measures, and provides additional background for one-dimensional optimal transport before considering one-dimensional signed optimal transport. Our work is summarized in Section~\ref{sec:con}.

\section{Background on Optimal Transport}\label{sec2}

We now provide the essential background on optimal transport that is need for the work that follows. 
General background material far beyond what we present here can be found in \cite{ambrosio2013user, mccann2011five, santambrogio2014introduction}.

Optimal transport considers probability measures on Polish spaces $(X,\mathscr{B}_X)$ and $(Y, \mathscr{B}_Y)$. 
Let the set of probability measures on $X$ be denoted by $\Pro(X)$. 
If $\mu\in \Pro(X)$ and $\nu \in \Pro(Y)$, then the optimal transport cost is the infimum of the functional $C(\gamma)$.
The functional $C(\gamma)$ is defined by 
\begin{equation}
\label{eq:kantorovich}
C(\gamma) = \int_{X\times Y} c(x,y)d\gamma(x,y)    
\end{equation}
where $\gamma \in \Pro(X\times Y)$ is a coupling of $\mu$ and $\nu$. 
A measure $\gamma\in \Pro(X\times Y)$ is a \textbf{coupling} of $\mu$ and $\nu$ if $(\proj_X)_{\#}\g = \mu$ and $(\proj_Y)_{\#} \g = \nu$, with $\proj_X$ and $\proj_Y$ being the projection maps from $X\times Y$ to $X$ and $Y$ respectively.
The set of couplings in $\Pro(X\times Y)$ is denoted $\Pi(\mu,\nu)$.
Any $\g \in \Pi(\mu,\nu)$ that obtains the infimal value is called an \textbf{optimal coupling} of $\mu$ and $\nu$.
The \textbf{cost function} $c(x,y)$ is assumed to be continuous and non-negative for reasons discussed in Section~\ref{sec3}.
So long as $c(x,y)$ is lower semi-continuous and bounded from below, then there exists an optimal coupling, as stated in Theorem~\ref{thm:kantdual}. 
Allow $C(\mu,\nu)$ to refer to the optimal transport cost between $\mu$ and $\nu$.

We now cite two important theorems on the structure of optimal couplings.

\begin{theorem}[Optimality is inherited by restriction {\cite[Pg. 46]{villani2008optimal}}]
\label{thm:restriction}
Let $(X,\mu)$ and $(Y,\nu)$ be two Polish spaces; let $a:X\mapsto \mathbb{R}~\cup\{-\infty\}$ and $b:Y\mapsto \mathbb{R}~\cup \{-\infty\}$ be two upper semi-continuous functions such that $a\in L^1(\mu)$, $b\in L^1(\nu)$.
Let $c:X\times Y\mapsto \mathbb{R}~\cup\{+\infty\}$ be a lower semi-continuous cost function, such that $c(x,y)\geq a(x)+b(y)$ for all $x,y$.
Let $C(\mu,\nu)$ be the optimal transport cost from $\mu$ to $\nu$. Assume that $C(\mu,\nu)<\infty$ and let $\gamma\in \Pi(\mu,\nu)$ be an optimal coupling. 
Let $\tilde\gamma$ be a non-negative measure on $X\times Y$ such that $\tilde \gamma\leq\gamma$ and $\tilde\gamma[X\times Y]>0$. 
Then the probability measure
\[
\gamma':= \frac{\tilde\gamma}{\tilde\gamma[X\times Y]}
\]
is an optimal coupling between its marginals. 

Moreover, if $\gamma$ is a unique optimal coupling between $\mu$ and $\nu$, then $\gamma'$ is a unique optimal coupling between its marginals.
\end{theorem}

Theorem \ref{thm:restriction} can be understood as saying that if there is an inefficiency in a small part of the optimal coupling (represented by $\tilde\gamma$), then the larger plan $\gamma$ cannot be optimal and efficient.
The proof, as given in \cite{villani2008optimal}, follows this line of reasoning.
Theorem \ref{thm:restriction} is used to prove the Theorem~\ref{thm:kantdual}, the Kantorovich Duality Theorem. The Kantorovich Duality Theorem establishes the structure of optimality, which is that the support must be \textbf{$c$-cyclically monotone}. 
Additionally, it is a model for Corollary~\ref{corr:rnrestriction}, which depends on Theorem \ref{thm:restriction} via Theorem~\ref{thm:kantdual} (Kantorovich Duality Theorem). 

\begin{definition}[{\cite[Pg. 28]{santambrogio2015optimal}}]
\label{def:cCM}
Let $X$ and $Y$ be arbitrary sets and $c:X\times Y\mapsto (-\infty,\infty)$ be a function. A subset $\Gamma\subset X\times Y$ is said to be \textbf{$c$-cyclically monotone}, or \textbf{$c$-CM}, if for any $N\in\mathbb{N}$, permutation $\sigma$, and family of points $(x_1,y_1),\ldots,(x_n,y_n)$ in $\Gamma$, we have
\begin{equation}
    \label{eq:c-cyclic}
\sum_{i=1}^{N}c(x_i,y_i)\leq \sum_{i=1}^N c(x_i,y_{\sigma(i)}).
\end{equation}
\end{definition}
This notion is important because when the support of a coupling is $c$-cyclically monotone (henceforth referred to as $c$-CM), then it is optimal.
This relates to Theorem~\ref{thm:restriction} because if there were a collection of points in the support that were not $c$-CM, then we could form a non-optimal measure via restriction. 
The Kantorovich Duality Theorem, {\cite[Pg. 70]{villani2008optimal}}, is a standard result of the field, establishes the relationship between optimality and $c$-cyclic monotonicity as well as many other structural properties of optimal transport.
For our purposes, we focus on the equivalence between having $c$-CM support and optimality.

\begin{theorem}[Kantorovich Duality {\cite[Pg. 70]{villani2008optimal}}]
\label{thm:kantdual}
Let $(X,\mu)$ and $(Y,\nu)$ be two Polish probability spaces and let $c:X\times Y\mapsto \mathbb{R}~\cup \{\infty\}$ be a semi-continuous cost function, such that 
$$
\forall (x,y)\in X\times Y, ~~~c(x,y)\geq a(x)+b(y)
$$
for some real-valued upper semi-continuous functions $a\in L^1(\mu)$ and $b\in L^1(\nu)$.
If $c$ is real-valued and the optimal cost $C(\mu,\nu)=\inf_{\gamma\in \Pi(\mu,\nu)}\int c ~d\gamma$ is finite, then there is a measurable $c$-cyclically monotone set $\Gamma\subset X\times Y$ (closed if $a,~b,~c$ are continuous) such that for any $\gamma\in \Pi(\mu,\nu)$ the following statements are equivalent:
\begin{enumerate}
    \item  $\gamma$ is optimal;
    \item $\gamma$ is $c$-cyclically monotone;
    \item $\gamma$ is concentrated on $\Gamma$.
\end{enumerate}
\end{theorem}

For the full statement of Theorem~\ref{thm:kantdual} and proof consult \cite{villani2008optimal} or another standard source.
One important remark from \cite{villani2008optimal} is that $a,~b,~c$ being continuous is sufficient to ensure that the support of $\gamma$ to be $c$-CM. The continuity of $a, b,$ and $c$ will be assumed for the remainder of the paper.


\section{Absolutely Continuous Measures Inherit Optimality}\label{sec3}

In this section we extend the capabilities of optimal couplings by showing that optimality is retained when considering absolutely continuous measures with respect to the optimal couplings or one of the marginals.
This allows us to send a Radon-Nikodym derivative `through' an optimal coupling akin to sending a vector through a matrix, or a measure through a transition kernel. 
Thus optimal transport couplings can behave similarly to Markov kernels and this will help us establish optimal transport kernels in Section~\ref{sec4}.

Theorem~\ref{thm:restriction} establishes that a measure $\gamma'$ inherits optimality when it is generated by renormalizing a measure $\tilde\gamma$ that is dominated by an optimal transport plan $\gamma$. 
This paper revisits and strengthens that relationship by showing that any absolutely continuous measure with respect to an optimal coupling will also be optimal, so long as it retains finite cost. 

To simplify our presentation we will only consider continuous and non-negative cost functions.
Further, we will assume that $c(x,y)$ satisfies the pointwise upper bound $c(x,y)\leq c_X(x)+c_Y(y)$ for fixed functions $c_X(x)$, and $c_Y(y)$. 
We then restrict our attention to measures $\mu \in \Pro(X)$ and $\nu\in\Pro(Y)$ such that $c_X(x)\in L(\mu)$ and $c_Y(y)\in L(\nu)$. 
Thus the $c_X$-integral of $\mu$ and $c_Y$-integral of $\nu$ are finite. 
We denote the set of such measures $\Pro_c(X)$ and $\Pro_c(Y)$. 
Limiting to $\Pro_c(X)$ and $\Pro_c(Y)$ is a common restriction. 

\begin{remark}
In an analogy to Wasserstein distances, we note that restricting to $\Pro_c(X)$ is equivalent to looking at the measures with finite \pth-moments for the \pth~Wasserstein distance where the cost function is given by $c(x,y) = \dist(x,y)^p$. 
\end{remark}

It is necessary to consider only continuous cost functions to ensure that the support of our optimal coupling is $c$-CM. 
We assume non-negativity to circumvent the need to prove the integrability of $a(x)$ and $b(y)$ with respect to the new marginals.
When the cost is non-negative we can set $a$ and $b$ to be 0, ensuring the integrability of $a$ and $b$.
Unlike the continuity of the cost, this is optional and the theory can be furthered by applying more specific assumptions. 

One such assumption could be restricting not only to measures which have finite $c_X$ (or $c_Y$ integrals) but also finite $a(x)$ (or $b(y)$) integrals.
If we were in a case where we have $-c_X(x)-c_Y(y)\leq c(x,y)\leq c_X(x)+c_Y(y)$, then we would immediately have integrability with respect to $a(x) = -c_X(x)$ and $b(y) = c_Y(y)$.

We prove the following corollary of Theorem~\ref{thm:kantdual} (Kantorovich Duality) to extend Theorem~\ref{thm:restriction} from saying that optimality is inherited by restriction to establishing that optimality is inherited by absolutely continuous measures, under appropriate hypotheses. 

\begin{corollary}
\label{corr:rnrestriction}
Let $(X,\mu)$ and $(Y,\nu)$ be two Polish probability spaces and let $c(x,y)$ be a non-negative continuous cost function. 
Let $\gamma$ be an optimal coupling between $\mu$ and $\nu$ with finite cost.
Let $\omega\in \Pro(X\times Y)$ be another probability measure that is absolutely continuous with respect to $\gamma$. 
Then if $\omega$ has finite cost, it is an optimal transport plan between its marginals, $\eta:=(\proj_X)_\# \omega \in\Pro(X)$ and $\zeta :=(\proj_Y)_\#\omega \in\Pro(Y)$ . 
\end{corollary}
\begin{proof}
Consider such an $\omega$.
By assumption it has finite cost, and hence it is an immediate upper bound for the optimal transport cost $C(\eta,\zeta)$, as defined in Equation \eqref{eq:kantorovich}.
Thus we satisfy the hypothesis for Theorem~\ref{thm:kantdual}. 

Since $\gamma$ is an optimal transport plan, it is supported on a $c$-CM set $\Gamma$.
Since $\omega$ is absolutely continuous with respect to $\gamma$, $\supp(\omega)\subset \Gamma$.
Thus, by Theorem~\ref{thm:kantdual}(b) it is an optimal transport plan between its marginals.  
\end{proof}

This is a consequence of the Kantorovich Duality Theorem (Theorem \ref{thm:kantdual}), but is an extension of Theorem \ref{thm:restriction}.

A useful application of Corollary \ref{corr:rnrestriction} is when we have an optimal transport plan and are looking at a measure that is absolutely continuous with respect to one of the marginals. 
In order to ensure that we retain finite transport cost, we need to add an additional constraint either on the Radon-Nikodym derivative of the measure, or on the behavior of the cost function on the support of the optimal transport plan. 
Corollary~\ref{corr:rnkernelA} proves the case when we place an additional constraint on the Radon-Nikodym derivative, and Corollary~\ref{corr:rnkernelB} proves the case with the additional constraint of the cost function on the support of the original optimal transport plan. 

\begin{corollary}
\label{corr:rnkernelA}
On a Polish space $X$, let $c(x,y)$ be a non-negative cost function satisfying the conditions of Theorem \ref{thm:kantdual}, with $c(x,y)\leq c_X(x)+c_Y(y)$.
Then let $\mu\in\Pro_c(X)$ and $\nu\in\Pro_c(Y)$, and $\gamma$ be an optimal coupling between them. 
If $\eta\in\Pro_c(X)$ and $\eta\ll\mu$ with $f(x):=\frac{d\eta}{d\mu}(x)$ and $f(x)< M$ for some $M$, then $\omega(x,y):=f(x)\gamma(x,y)$ is an optimal coupling with marginals $\eta(x)$, and $\zeta(y):=(\proj_y)_\#\omega(x,y)$. 
\end{corollary}

\begin{proof}
This is proven by first showing that $\omega$ has finite transport cost, and then showing that $\omega$ is optimal because it is absolutely continuous with respect to $\gamma$ by applying Corollary~\ref{corr:rnrestriction}. 

Observe that $\omega$ has finite transport cost since 
\begin{eqnarray*}
\int_{X\times Y}c(x,y)d\omega(x,y)&\leq& \int_{X\times Y} (c_X(x)+c_Y(y))f(x)d\gamma(x,y)\\
&\leq& \int_{X\times Y} c_X(x)f(x)d\gamma(x,y)+\int_{X\times Y} c_Y(y)M d\gamma(x,y)\\
&=& \int_X c_X(x)f(x)d\mu(x)+M\int_Y c_Y(y)d\nu(y)\\
&=& \int_X c_X(x)d\eta(x)+M\int_Y c_Y(y)d\nu(y)<\infty.
\end{eqnarray*}

Clearly $\omega\ll\gamma$, with Radon-Nikodym derivative $\frac{d\omega}{d\gamma}(x,y)=f(x)$. Thus $\omega$ satisfies the conditions of Corollary \ref{corr:rnrestriction} and is an optimal transport plan between the marginals $\eta$ and $\zeta$.
\end{proof}

\begin{corollary}
\label{corr:rnkernelB}
On a Polish space $X$, let $c(x,y)$ be a cost function satisfying the conditions of Theorem \ref{thm:kantdual}, with $c(x,y)\leq c_X(x)+c_Y(y)$.
Then let $\mu\in\Pro_c(X)$ and $\nu\in\Pro_c(Y)$, and $\gamma$ be an optimal coupling between them.
Let $\Gamma$ be a $c$-CM set on which $\gamma$ is concentrated. 
If $\eta\in\Pro_c(X)$ and $\eta\ll\mu$ with $f(x):=\frac{d\eta}{d\mu}(x)$ and $c(x,y)\leq A c_X(x)+B$ on $\Gamma$ for some $A,B < \infty$ then $\omega(x,y):=f(x)\gamma(x,y)$ is an optimal coupling with marginals $\eta(x)$, and $\zeta(y):=(\proj_y)_\#\omega(x,y)$.
Additionally, $\zeta(y)\in\Pro(Y)$ has finite $c_Y$-integral.
\end{corollary}

\begin{proof}
This is proven by first showing that $\omega$ has finite transport cost, and then showing that $\omega$ is optimal because it is absolutely continuous with respect to $\gamma$ by applying Corollary~\ref{corr:rnrestriction}. 

Observe that $\omega$ has finite transport cost since 
\begin{eqnarray*}
\int_{X\times Y}c(x,y)d\omega(x,y)&=& \int_{\Gamma} (c(x,y))f(x)d\gamma(x,y)\\
&\leq& \int_{\Gamma} (A c_X(x)+B)f(x)d\gamma(x,y)\\
&=& \int_X (A c_X(x)+B)f(x)d\mu(x)\\
&=& A\int_X c_X(x)d\eta(x)+B<\infty.
\end{eqnarray*}

Clearly $\omega\ll\gamma$, with Radon-Nikodym derivative $\frac{d\omega}{d\gamma}(x,y)=f(x)$. Thus $\omega$ satisfies the conditions of Corollary \ref{corr:rnrestriction} and is an optimal transport plan between the marginals $\eta$ and $\zeta$. 

Finally, we show that $\zeta(y)$ has finite $c_Y$-integral by noting that 
\begin{eqnarray*}
   \int_Y c_Y(y)d\zeta &=& \int_{X\times Y}c_Y(y)d\omega(x,y)\\
   &=&\int_{\Gamma}c_Y(y)d\omega(x,y)\\
   &<&\int_{\Gamma} Ac_X(x)+B d\omega(x,y)\\
   &=&\int_{X\times Y} Ac_X(x)+B d\omega(x,y)\\
   &=& \int_X Ac_X(x)+B d\eta(x)<\infty
\end{eqnarray*}
\end{proof}

Corollaries \ref{corr:rnkernelA} and \ref{corr:rnkernelB} are two complementary means of ensuring the new coupling $\omega$ has finite cost.
Note, that only Corollary \ref{corr:rnkernelB} ensures us that $\zeta(y) = (\proj_Y)_\#(f(x)\gamma(x,y))$ has finite $c_Y$-integral.

Based on the results of the section, we have shown that we can use existing optimal couplings to guarantee the optimality of absolutely continuous couplings (when finite). 
This led to two corollaries that allow us to send an absolutely continuous marginal `through' an optimal coupling. 
These corollaries show that we can treat optimal couplings like Markov Kernels, which we will show in the next section. 

\section{Markov Kernels for Optimal Transport}\label{sec4}
In this section we connect our previous work to the functionality of transition kernels. 
We provide an example when both measures are discretely supported to provide intuition before providing background on transition kernels and extending optimal transport theory to Markov kernels. 

Corollary~\ref{corr:rnrestriction} allows us to consider any $\omega\in\Pro(X\times Y)$ that is absolutely continuous with respect to an optimal coupling $\gamma$, but there is a reason to focus on the case presented in Corollaries \ref{corr:rnkernelA} and \ref{corr:rnkernelB}. 
Generally $\frac{d\omega}{d\gamma} = g(x,y)$ is a function depending on both variables, but in Corollaries \ref{corr:rnkernelA} and \ref{corr:rnkernelB}, we focus on the cases when $g(x,y)$ depends only on one variable. 
Doing so allows us to view the coupling as acting as a stochastic kernel for measures absolutely continuous with respect to the marginal as opposed to the coupling. 

We show this reframing when $\mu$ and $\nu$ are discrete before describing it more generally. 
\subsection{Discrete Optimal Transport and Stochastic Matrices}\label{s3ssec1}
Consider optimal transport between probability vectors $ {\mu}$, and $ {\nu}$, and a cost function $c(x,y)\leq c_X(x)+c_Y(y)$.
The optimal coupling between $ {\mu}=(\mu_i)$, and $ {\nu}=(\nu_j)$ is then a $c$-CM array $\gamma=(\gamma_{i,j})$ such that 
\begin{equation}
\label{eq:marginalization}
    \mu_i=\sum_j \gamma_{i,j}, \text{~and~}\nu_j=\sum_i \gamma_{i,j}.
\end{equation}
Equation \eqref{eq:marginalization} is  called the marginalization condition and is often written as $\mu=\gamma \mathds{1}$, and $\nu^T=\mathds{1}^T\gamma$, although $\gamma$ is not used as a matrix in any other context. 
Note that $\mu$ is the sum of the columns of $\gamma$, while $\nu$ is the sum of the rows. 

If we assume $ {\mu}$ and $ {\nu}$ have all non-zero entries (since the corresponding row or column of $\gamma$ would otherwise be all zeros), then we can form the array 
\begin{equation}
\label{eq:stochmatrix}
S_{i,j}=\frac{\gamma_{i,j}}{\nu_j}.
\end{equation}
Summing the entries of each column of $S$, we have
\begin{eqnarray}
\label{eq:stochcomputation}
\sum_i S_{i,j}&=&\sum_i \frac{\gamma_{i,j}}{\nu_j}\nonumber\\
&=&\frac{1}{\nu_j}\sum_i \gamma_{i,j}\nonumber\\
&=& \frac{1}{\nu_j}\nu_j=1. 
\end{eqnarray}
This shows that $S$ is a stochastic matrix, and by construction we have $$
S\nu = \sum_j \frac{\gamma_{i,j}}{\nu_j}\nu_j=\mu.
$$
We can form a stochastic matrix that sends $\mu$ to $\nu$ similarly.
Because $\mu$ and $\nu$ represent measures, it is more suitable to treat them as row vectors moving forward and to view it as $\tilde S = (\gamma_{i,j}/\mu_i)$, and $\mu \tilde S = \nu$. 
The vectors $\mu$ and $\nu$ were chosen to be non-zero to avoid dividing by 0, however by restricting to the support of the measures, we avoid any difficulty. 
In the row (column) corresponding to a zero value of $\mu$ ($\nu$), all entries will be zero since we are summing non-negative values to 0. 
This point will be revisited when we extend this concept from vector measures to the more general case. 

The construction of the stochastic matrix $\tilde S$ becomes useful when considering a measure $\eta$ that is absolutely continuous with respect to $\mu$. 
Every measure on the same space as $\mu$ will be absolutely continuous with respect to $\mu$ since it is non-zero everywhere. Additionally, the Radon-Nikodym derivative will be bounded by the inverse of the smallest entry in $\mu$.

Then defining $\zeta := \eta \tilde S$, we have the that the coupling
\begin{equation}
    \label{eq:discreteRNcoupling}
    (\omega_{i,j}) = (\frac{\gamma_{i,j}}{\mu_i}\eta_i),
\end{equation}
is optimal by applying Corollary~\ref{corr:rnkernelA} with $\gamma$ and $f(i) = \frac{\eta_i}{\mu_i}$, which is the RN-derivative of $\eta$ with respect to $\mu$. 
Here we begin to extend the capabilities of an optimal coupling to include some of advantages of a transport map.
This allows us to take advantage of prior efforts to find optimal couplings while allowing those couplings to be used like a transport map.  

We started with the case of discrete measures not only to provide the context of stochastic matrices as a means of understanding the work in the remainder of this section, but also to provide insight into how these ideas can be implemented. 
\subsection{Optimal Transport Using Markov Kernels}\label{s3ssec2}

Markov kernels are a well-studied class of transition kernels. 
In this section we use Markov kernels in a novel way as a tool for optimal transport. 
\begin{definition}\cite[Page 37]{ccinlar2011probability} Let $(X,\mathfrak{M})$ and $(Y, \mathfrak{N})$ be measurable spaces and $K$ a mapping from $X\times \mathfrak{N}$ into $[0,\infty]$. 
Then $K$ is called a \textbf{transition kernel} from $(X,\mathfrak{M})$ into $(Y,\mathfrak{N})$ if:
\begin{enumerate}
    \item The function $x\mapsto K(x,B)$ is $\mathfrak{M}$-measurable for every subset $B$ in $\mathfrak{N}$, and
    \item the mapping $B\mapsto K(x,B)$ is a measure on $(Y,\mathfrak{N})$ for every $x\in X$. \label{list:kernel2}
\end{enumerate}
If in item 2 the measure is a probability measure for all $x\in X$, then $K$ is called a \textbf{stochastic kernel} or a \textbf{Markov kernel}.
\end{definition}
The definition of a kernel tells us that if we fix a measurable set in the space $Y$, then we have a measurable function on $(X,\mathfrak{M})$ and if we fix a point in the space $X$, then we have a measure on $Y$. 
As of yet, this does not provide us the functionality of mapping measures to measures, however it is a simple and standard construction to see that functionality. 

\begin{theorem}{\cite[Pg. 38]{ccinlar2011probability}}
\label{thm:transkernel}
Let K be a transition kernel from $(X,\mathfrak{M})$ into $(Y,\mathfrak{N})$. Then 
$$
Kg(x)=\int_Y g(y) K(x,dy),~x\in X,
$$
defines a measurable function $Kg$ in $\mathfrak{M}$ for every measurable function $g$ in $\mathfrak{N}$;
$$
\mu K (B)=\int K(x,B)d\mu(x), ~B\in \mathfrak{N},
$$
defines a measure $\mu K$ on $(Y,\mathfrak{N})$ for each measure $\mu$ on $(X,\mathfrak{M})$; and 
$$
(\mu K) g = \mu(Kg)=\int_X d\mu(x)\int_Y g(y) K(x,dy)
$$
for every measure $\mu$ on $(X,\mathfrak{M})$ and $\mathfrak{N}$-measurable function $g(y)$ on $Y$.
\end{theorem}
Theorem \ref{thm:transkernel} allows us to treat kernels as objects that map measures. 

In addition to viewing the kernel which maps a measure $\mu$ on $X$ to a measure $\mu K$ on $Y$, we can also view it as mapping to a measure on the product space $X\times Y$, with $\sigma$-algebra $\mathfrak{M}\otimes\mathfrak{N}$.
This construction will be similar to vector-matrix multiplication if we were to halt the process before summing along the rows (or columns for row vectors). 

\begin{theorem}{\cite[Pg. 41]{ccinlar2011probability}}
\label{thm:productmeas}
Let $\mu$ be a measure on $(X,\mathfrak{M})$ and $K$
be a Markov kernel from $(X,\mathfrak{M})$ to $(Y,\mathfrak{N})$. 
Then define the measure $\gamma$ by how it acts on measurable functions $f(x,y)$ in $(\mathfrak{M}\otimes\mathfrak{N})_+$ in the following way:
\begin{equation*}
    \gamma f=\int_X d\mu(x)\int_Y f(x,y) K(x,dy).
\end{equation*}
This defines a measure on the product space $(X\times Y,\mathfrak{M}\otimes\mathfrak{N})$. 
If $K$ is a Markov kernel and $\mu$ is $\sigma$-finite, then $\gamma$ is $\sigma$-finite and is the unique measure on the product space satisfying
$$
\gamma(A\times B)= \int_A K(x,B)d\mu(x),~ A\in\mathfrak{M}, B\in \mathfrak{N}.
$$
\end{theorem}

When we wish to discuss $\mu K$ as a measure on the joint space $(X\times Y)$, we will write $\mu\odot K$ to distinguish it from the image measure $\mu K$ on $(Y,\mathfrak{N})$. 
This is meant to be reminiscent of the element-wise multiplication between matrices.

With these properties in place, we are able to usefully talk about the concept of an optimal transport kernel. 
\begin{definition}
A Markov kernel $K$ is an \textbf{optimal transport kernel} for a cost function $c(x,y)$ if for measures $\mu$ and $\nu:=\mu K $ the product measure $\gamma:=\mu\odot K$ is an optimal coupling between $\mu$ and $\nu$ when the transport problem has finite cost.
\end{definition}

To ensure that a transport kernel gives rise to a coupling with a finite cost, it is useful to consider cost functions $c(x,y)\leq c_X(x)+c_Y(y)$, as well as to consider transition kernels that are not only Markov kernels, but also \emph{$c_X$-bounded Markov kernels}, as defined by
\begin{definition}
A Markov transition kernel is said to be a $c_X-$\textbf{bounded} kernel if 
$$
\int_Y c_Y(y)K(x,dy)< Ac_X(x)+B \quad\text{for all $x\in X$}.
$$
\end{definition}

When $\mu\in\Pro_c(X)$ and $K$ is a $c_X$-bounded kernel, then we can be assured that $\mu\odot K$ has finite transport cost and that $\mu K\in \Pro_c(Y)$ as shown in the following theorem.
\begin{theorem}
Let $\mu\in\Pro_c(X)$ and $K$ be a $c_X$-bounded kernel. 
Then the product measure $\gamma := \mu \odot K$ has finite transport cost and $\mu K \in \Pro_c(Y)$. 
If $K$ is additionally an optimal transport kernel, then $\gamma$ is an optimal coupling between $\mu$ and $\mu K$.
\end{theorem}
\begin{proof}
We show that $\mu\odot K$ has finite cost by definition of $c_X$-bounded. 

\begin{eqnarray}
\mu\odot K (c(x,y))&=&\int_Xd\mu(x)\int_Y c(x,y)K(x,dy)\\
                &\leq& \int_Xd\mu(x)  \int_Y c_X(x)+c_Y(y)K(x,dy)\\
                &=& \int_Xd\mu(x) c_X(x)+\int_Y c_Y(y)K(x,dy) \\
                &<& \int_X (A+1) c_X(x)+B d\mu(x) < \infty.
\end{eqnarray}
The final inequality arises because $\mu$ is in $\Pro_c(X)$. This shows that $\mu\odot K$ is a finite transport plan between $\mu$ and $\mu K$.

Similarly, 
\begin{eqnarray}
\mu K (c_Y(y))&=&\int_Xd\mu(x)\int_Y c_Y(y)K(x,dy)\\
            &=& \int_Xd\mu(x)\int_Y c_Y(y)K(x,dy) \\
                &<& \int_X Ac_X(x)+B d\mu(x) < \infty,
\end{eqnarray}
which shows that $\mu K \in \Pro_c(Y)$. 

When $K$ is an optimal transport kernel, the coupling $\mu \odot K$ is optimal so long as it has finite cost, which we just showed. 
\end{proof}

As alluded to earlier, we have been (nearly) constructing optimal transport kernels whenever we constructed an optimal transport plan.
By utilizing Corollary~\ref{corr:rnkernelB}, we can use a optimal coupling as an optimal transport kernel for any measure that is absolutely continuous with respect to a marginal. 
We begin by defining the set of measures for which an optimal coupling may act as an optimal transport kernel. 

\begin{definition}
Let $\ACP^b_c(\mu)$ be the set of measures in $\Pro_c(X)$ which are absolutely continuous with respect to $\mu$ with bounded Radon-Nikodym derivatives. 
\end{definition}

We now show how to use a coupling with marginal $\mu$ as a transition kernel for the set of measures $\ACP^b_c(\mu)$.
\begin{definition}
Let $\gamma$ be a coupling between measures $\mu$ and $\nu$. Then we define $K^\gamma$ as the transition kernel associated to $\gamma$ defined on the measures $\eta\in\ACP^b_c(\mu)$ by
$$
\eta \odot K^\gamma = \frac{d\eta}{d\mu}(x)\gamma(x,y)
$$
and
$$
\eta K^\gamma = (\proj_Y)_\# \left(\frac{d\eta}{d\mu}(x)\gamma(x,y)\right).
$$
\end{definition}

\begin{theorem}
If $\gamma$ is an optimal coupling between $\mu$ and $\nu$ in $\Pro_c(X)$, $\Pro_c(Y)$ respectively, then on $\ACP^b_c(\mu)$ $K^\gamma$ is an optimal transport kernel and $\eta K^\gamma$ is in $\Pro_c(Y)$ for any $\eta\in \ACP^b_c(\mu)$ . 
\end{theorem}
\begin{proof}
The proof follows from the definitions we established. 

Let $\eta\in\ACP^b_c(\mu)$.
Then 
\begin{eqnarray}
\eta\odot K^\gamma (c(x,y))&=&\int_X \int_Y c(x,y)\frac{d\eta}{d\mu}(x)\gamma(x,y)\\
                &\leq& M \int_X\int_Y c(x,y)\gamma(x,y)  < \infty.
\end{eqnarray}

This shows that $\eta \odot K^\gamma$ has finite cost and is thus the optimal coupling between $\eta$ and $\eta K^\gamma$.

To show that $\eta K^\gamma$ is in $\Pro_c(Y)$, we need only note that 
\begin{eqnarray}
    \int_Yc_Y(y) \eta K^\gamma(y) &=& \int_X\int_Y c_Y(y) \frac{d\eta}{d\mu}(x) \gamma(x,y)\\ &\leq& M\int_X\int_Y c_Y(y)\gamma(x,y) = M\int_Yc_Y(y)d\nu<\infty.\
\end{eqnarray}

This completes both parts of the proof. 
\end{proof}

For the remainder of the paper we will talk about the kernel associated to a coupling and omit discussion of the limited range of the kernel.
It is left to future work to extend the kernel to be defined point-wise like how regular transition kernels are defined, and to extend them beyond the support of $\gamma$ when that support is limited. 
Additionally, understanding if there are better ways to extend beyond the measures that are absolutely continuous with bounded RN-derivatives. 

\subsection{$c$-Cyclic Monotonic Compatibility}\label{subsec3}
Here, we define $c$-cyclic monotonic compatibility, which will be a tool to recognize when two couplings can be combined to form a single optimal transport kernel. 

Given an optimal transport kernel $K$, and two measures $\mu_1$ and $\mu_2$ for which $\mu_1\odot K$ and $\mu_2\odot K$ are optimal couplings, then we would expect there to be some compatibility between the optimal couplings they generate. 
This is indeed the case. 
Additionally, the same notion tells us when we can take two optimal couplings and view them as arising from the same optimal transport kernel. 
We call this notion \textbf{$c$-CM compatibility}.
In this section, we discuss the supports of various measures.
We remind the reader that we are working with Polish spaces, and so the topology is defined by the distance metric. 

\begin{definition}
Two optimal couplings $\gamma_1$ and $\gamma_2$ are \textbf{$c$-CM compatible} if $ \supp(\gamma_1)~\cup\supp{(\gamma_2)} $ is c-cyclically monotone. 
\end{definition}

The following two theorems establish that optimal couplings from the same kernel will be 
$c$-CM compatible and a partial result saying that we can view $c$-CM couplings as coming from the same kernel. 

\begin{theorem}
Let $K$ be an optimal transport kernel and let $\gamma_1 := \mu_1\odot K$ and $\gamma_2 : = \mu_2\odot K $ for some $\mu_1$ and $\mu_2$. Then $\gamma_1$ and $\gamma_2$ are $c$-CM compatible. 
\end{theorem}
\begin{proof}
Consider $\mu_3 = \frac12(\mu_1+\mu_2)$ and let $\gamma_3 = \mu_3 \odot K$. The supports of $\gamma_1$ and $\gamma_2$ are contained with the support of $\gamma_3$, so 
\begin{equation*}
\supp(\gamma_1)~\cup\supp(\gamma_2) \subset \supp(\gamma_3).
\end{equation*}
The support of $\gamma_3$ is $c$-CM since $K$ is an optimal transport kernel (and we're assuming that $c$ is a continuous cost function). Thus $\supp(\gamma_1)~\cup\supp(\gamma_2)$ is $c$-CM.

\end{proof}

Showing that when two couplings are $c$-CM compatible then they can be formed from the same optimal transport kernel follows a similar construction. 

\begin{theorem}
\label{thm:cCMprod}
Let $\gamma_1$ and $\gamma_2$ be optimal couplings that are $c$-CM compatible with disjoint $X$-support, which is defined as $X-\supp(\gamma) = \supp((\proj_X)_\# \gamma_i)$. 
Then there exists an optimal transport kernel $K$ and measures $\mu_1$ and $\mu_2$ such that $\gamma_1 = \mu_1 \odot K$ and $\gamma_2 = \mu_2 \odot K$. 
\end{theorem}
\begin{proof}
Let $\gamma_3 = \frac12 (\gamma_1+\gamma_2)$. This is an optimal coupling due to the assumption of $c$-CM compatibility. 
Notice that $\gamma_1$ and $\gamma_2$ are absolutely continuous with respect to $\gamma_3$. Let $f_i(x,y)$ be the RN-derivative of $\gamma_i$ with respect to $\gamma_3$.
Since $\gamma_1$, and $\gamma_2$ have disjoint $X$-support, their supports are disjoint.
Thus $f_i = 2$ on the support of $\gamma_i$ and $f_i = 0$ otherwise.

Let $\mu_i = (\proj_X)_\#\gamma_i$ for $i=1,2,3$. 
Then because the $X$-supports are disjoint, the $m_i(x) := \frac{d\mu_i}{d\mu_3} = 2$ on the support of $\mu_i$.
Let $A_i$ be the support of $\mu_i$ and $B_i$ be the support of $\gamma_i$. 
Notice that $m_i(x)\restriction_{B_3}\equiv f_i(x,y)$. 
Thus $m_i(x)\gamma_3(x,y) = f_i(x,y)\gamma_3(x,y)$. 

Let $K^\gamma$ be the optimal transport kernel associated with $\gamma_3$. 
Consider $\mu_i\odot K^\gamma$.
Notice that 
\begin{eqnarray}
    \mu_i\odot K^\gamma &=& m_i(x)\gamma_3(x,y) \\
    &=&   f_i(x,y)\gamma_3(x,y) =\gamma_i(x,y) 
\end{eqnarray}
Thus $\mu_i \odot K^\gamma = \gamma_i$.
\end{proof}

It is necessary for the supports of the measures $\mu_1$ and $\mu_2$ to be disjoint in $X$ if we want to realize $\gamma_1$ and $\gamma_2$ them as product measures from the same kernel. 
To see this, consider the optimal transport plans $\gamma_1$ and $\gamma_2$ that send $\delta_x$ to $\delta_{y_1}$ and $\delta_x$ to $\delta_{y_2}$. 
These are $c$-CM compatible, but it is impossible for there to be an appropriate kernel.

However, there is of course still a larger class of $c$-CM optimal transport plans that we can realize as product measures from the same kernel. 
\begin{theorem}
Consider two $c$-CM compatible optimal couplings $\gamma_1$ and $\gamma_2$ such that 
$$
\gamma_1(x,y) = a\pi_1(x,y)+f(x)\pi_2(x,y)$$
and 
$$\gamma_2 = g(x)\pi_2(x,y)+b\pi_3(x,y),$$
for probability couplings $\pi_1$, $\pi_2$ and $\pi_3$ that have mutually disjoint $X$-support and functions $f(x),g(x)$ which are zero outside of the support of $\pi_2$. 
Then there exists an optimal transport kernel $K$ and measures $\mu_1$ and $\mu_2$ such that $\mu_1\odot K = \gamma_1$ and $\mu_2\odot K = \gamma_2$.
\end{theorem}
\begin{proof}
Each $\pi_i$ is an optimal coupling from Corollary~\ref{corr:rnrestriction}. 
The coupling $\gamma = \frac{1}{3}(\pi_1+\pi_2+\pi_3)$ is an optimal coupling due to the $c$-CM compatibility, and let 
$\mu = (\proj_X)_\#\gamma$.

Let $K^\gamma$ be the optimal transport kernel associated with $\gamma$. 
Let $A=\supp(\pi_1)$, $B = \supp(\pi_2)$ and $C = \supp(\pi_3)$, and $A_x, B_x, C_x$ the corresponding $X$-supports.
Let $\eta_i$ be the $(\proj_X)_\#\pi_i$.

Then as in Theorem~\ref{thm:cCMprod}, we utilize the fact that a function supported on $A_x$ becomes a function support on $A$ when we restrict it to the support of $\gamma$, and likewise for $B_x$ and $C_x$.
Then, let $\mu_1 = (a\eta_1+f(x)\eta_2)$.
Notice that $\frac{d\mu_1}{d\mu} = 3a\mathds{1}_{A_x}+3f(x)$.
Then observe that
\begin{eqnarray}
   \mu_1\odot K^\gamma &=& (3a\mathds{1}_{A_x}+3f(x))(\frac{1}{3}(\gamma_1+\gamma_2+\gamma_3))\\
   &=& a\gamma_1+f(x)\gamma_2 = \gamma_1.
\end{eqnarray}
This shows that $\mu_1\odot K^\gamma = \gamma_1$.
Similarly, by setting $\mu_2 = g(x)\eta_2 + b\eta_3$, we will obtain that 
\begin{equation}
    \mu_2 \odot K^\gamma  =\gamma_2.
\end{equation}
\end{proof}

The work in this section shows how situations with multiple couplings may be simplified by viewing the couplings as arising from one transport kernel. 
In order to recognize such situations, we defined the notion of $c$-CM compatibility, which does a good job at recognizing this relationship. 
\subsection{A Note on Geodesics}
An interesting relationship between optimal transport and stochastic processes that deserves to be further explored is the relationship between the family of transition kernels generated by a geodesic in Wasserstein space and the Chapman-Kolmogorov equation for stochastic processes. 
Because path-lines of geodesics may cross in Wasserstein spaces when $p\geq2$, the kernels associated will form an inhomogenous stochastic process.

\section{Insights into Signed Optimal Transport using Kernels}\label{sec5}

In most modern formulations of optimal transport, the marginals and optimal coupling are all positive measures. 
While other frameworks examine the problem in different ways, positivity remains a necessary feature.
Optimal transport involving signed quantities is comparatively new, but is an active area \cite{ambrosio2011gradient,mainini2012description,piccoli2013transport,piccoli2014generalized,piccoli2016properties,piccoli2019wasserstein}. 
For example, applications have arisen in seismic imaging \cite{engquist2013application,engquist2016optimal,engquist2018seismic} and for modelling signed vortices \cite{ambrosio2011gradient}. 
However each application has used an ad hoc approach to make the signed measure into a positive one, e.g. exponentiating the density, adding a large constant to make the density positive, and treating the positive and negative parts separately.

Optimal transport kernels are a natural tool for a unified treatment of optimal transport for positive and signed measures.
While the theory developed in this section is more restrictive than some of the previous approaches, it provides a strong connection between the theories for the optimal transport of positive measures and of signed measures. 

\subsection{Signed Optimal Transport}
In this section we denote signed measures with Latin letters, e.g. $a,~b$ and the associated positive absolute value measures either as $\| a\|,~\| b\| $ or with Greek letters, i.e. $\alpha,~\beta$. 
Let $M(X)$ denote the space of signed measures over a Polish space $X$. 
The class of signed measures that we focus on are those $a\in M(X)$ with the following properties
\begin{enumerate}
    \item $\int_X da < \infty$ (finite integral),
    \item $\int_X d\alpha<\infty$ (finite mass),
    \item $\int_X c_X(x)  d\alpha(x)<\infty$ (finite moment).
\end{enumerate}
Analogous requirements will hold for measures $b\in M(Y)$.
We will typically consider the case when $X$ and $Y$ are two copies of the same space, but continue to label them as $X$ and $Y$ for clarity.
The requirements on the signed measures $a$ and $b$ ensure that there is an optimal transport coupling with finite cost between the corresponding positive measures $\alpha$ and $\beta$.

We will be looking at optimal couplings between two positive measures that have the same mass, but are not necessarily probability measures as the total mass may no longer be 1.
They must still satisfy the marginalization constraints and have $c$-CM support, but the marginals are no longer required to be probability measures.

\begin{definition}
A signed coupling $g$ is an \textbf{optimal signed coupling} between the marginals $a$ and $b$, which have finite integral, mass, and moment if it is equal to the product measure derived from $aK$ for some optimal transport kernel $K$.
\end{definition}

For two signed measures $a$ and $b$ to be connected by an optimal transport kernel,
they must have the same integral. 
This is because if $b=aK$, then 
$$
b(Y)=\int_X K(x,Y)da(x)=\int_X da(x),
$$
since $K$ is a Markov kernel and $K(x,Y)=1$ for all $x\in X$.
While $a$ and $b$ do not need to have the same mass, we will treat this as the standard.  
Kernels which do not preserve mass are considered special. 

A salient feature and restriction of the kernel based approach to signed optimal transport comes in the following theorem.

\begin{theorem}
\label{thm:kernelsandsignedmeasures}
Let $a$ and $b$ be two signed measures with equal mass and integral, and each with finite moment.
Let $a=a^+-a^-$ and $b = b^+ - b^-$ be the Jordan decompositions of $a$ and $b$ respectively. 
Then $a$ and $b$ are connected by an optimal transport kernel if and only if there are compatible optimal couplings $\gamma^+$, between $a^+$ and $b^+$, and $\gamma^-$, between  $a^-$ and $b^-$.
\end{theorem}
\begin{proof}
Let $\alpha=\|a\|=a^++a^-$ and $\beta=\|b\|=b^++b^-$.
Suppose that $a$ and $b$ are connected by an optimal transport kernel $K$, i.e. $aK=b$.

Since $a$ and $b$ have the same integral, $$\int_Xda=\|a^+\|-\|a^-\|=\|b^+\|-\|b^-\|=\int_Xdb,$$ and the same mass, $$\int_X d\alpha = \|a^+\|+\|a^-\|=\|b^+\|+\|b^-\|=\int_X d\beta,$$ 
then 
$$\|a^+\|=\|b^+\| \text{ and } \|a^-\|=\|b^-\|.$$

We want to first show that $a^+K=b^+$ and $a^-K=b^-$.
Notice that for any set $B$, 
\begin{equation}
\begin{split}
aK(B)=\int_X K(x,B)da&=\int_X K(x,B)da^+-\int_X K(x,B)da^- \\
&\leq \int_X K(x,B)da^+=a^+K(B).  
\end{split}
\end{equation}
Therefore $(aK)^+\leq (a^+K)^+$.
Now, since $aK=b$, we have $(aK)^+=b^+$.
We also already have that  $a^+K$ is a positive measure, so $(a^+K)^+=a^+K$.  Since $K$ is a Markov kernel, we have  $\|a^+K\|=\|a^+\|$, implying that
$\|a^+K\|=\|b^+\|$. 
Hence $(aK)^+=(a^+K)^+$ and  $b^+=a^+K$, that is, when $b$ is the image of $a$ from the kernel $K$, then $b^+=a^+K$. 

We note that in general, $b^+\leq a^+K$, but here we have equality since the measures have equal mass. Similarly we find $b^- = a^- K$, while in general $b^-\leq a^-K$.  
This yields  $a^+\odot K=\gamma^+$ and $a^-\odot K=\gamma^-$.
These measures will be compatible because $\alpha\odot K$ is an optimal coupling, and $\supp(\alpha\odot K)=\supp(\gamma^+)\cup \supp(\gamma^-)$. 

For the other direction, suppose that $\gamma^+$ and $\gamma^-$ are compatible and let $\gamma:=\gamma^++\gamma^-$. 
Observe that $\gamma$ is an optimal coupling between its marginals. 

Now $a^+\perp a^-$ and both are absolutely continuous with respect to $\alpha:=(\proj_x)_\#(\g)$. 
Let $K^\g$ be the transition kernel associated to $\g$, defined on measures that are absolutely continuous with respect to $\alpha$. 

Let $f_+(x)$ be the Radon-Nikodym derivative of $a^+$ with respect to $\alpha$, which will be equal to $+1$ on $\supp(a^+)\subset\supp(\alpha)$ and $0$ on $\supp(a^-)$ since $a^+\perp a^-$.
Let $f_-$ be the Radon-Nikodym derivative of $a^-$ with respect to $\alpha$, and similar statements will hold.
Then $f_+(x)\g=\gamma^+$ and $f_-(x)\g=\g^-$, so 
\begin{equation}
    aK^\g=(\proj_y)_\#\left((f_+(x)-f_-(x))\g\right)=(\proj_y)_\#(\gamma^+-\gamma^-)=b.
\end{equation}
In this way, $\g$ acts as a kernel sending $a$ to $b$.
\end{proof}

Theorem \ref{thm:kernelsandsignedmeasures} tells us that we will not be able to connect any two arbitrary signed measures together, only the ones where the positive and negative couplings are $c$-CM compatible.
This is restrictive for some versions of signed optimal transport that want to be able to connect any two signed measures in the same way as probability measures, but it is an inherent limitation from the approach focusing on $c$-CM as the salient feature of optimality.

Unfortunately, in two dimensions and higher it is not possible to partition the space of signed measures into classes such that there exists a kernel between all measures within a class and that any target for which there is a kernel is an element of the class.
We demonstrate this with the following example.
\begin{example}
\label{ex:kernhigherdims}
Let $X=\mathbb{R}^2$, and
$a_1=\delta_{(1,.5)}-\delta_{(-1,-.5)}$,
$a_2=\delta_{(-1,.5)}-\delta_{(1,-.5)}$, and
$a_3=\delta_{(0,1)}-\delta_{(0,-1)}$.
\end{example}
In this example there are optimal transport kernels $K_1$ and $K_3$ such that $a_1K_1=a_3$ and $a_3K_2=a_2$.
However, Theorem \ref{thm:kernelsandsignedmeasures} says that there is no kernel sending $a_1$ to $a_2$. 
This example more generally demonstrates that signed measures that in two dimensions or higher are not transitive. 
However they are in one dimension, and there is a simple tool that we can use to signify when two signed measures are in the same equivalence class. 
\subsection{Review of One-Dimensional Optimal Transport}

Here, we present the components of one-dimensional optimal transport theory that are directly needed for our work on one-dimensional signed measures.
Interested readers can consult Chapter 2 of \cite{santambrogio2015optimal} for a full treatment of this subject.

\begin{theorem}
{\cite[Pg. 60]{santambrogio2015optimal}.}
\label{prop:pseudoinvpush}
If $\mu\in \Pro(\mathbb{R} )$ and $G_\mu$ is the pseudo-inverse of its CDF $F_\mu$, then $(G_\mu)_\#(\mathscr{L}\restriction [0,1])=\mu$, where $\mathscr{L}$ is the Lebesgue measure.
\end{theorem}
The psuedo-inverse is an important tool for us to construct a particular coupling. 

\begin{definition}[{\cite[Pg. 61]{santambrogio2015optimal}}]
We call the coupling $(G_\mu,G_\nu)_\# (\mathscr{L}\restriction[0,1])$
the \textbf{monotone coupling} and denote it by $\gamma_{mon}$. 
\end{definition}

The monotone coupling has the property that 
\begin{equation*}
\gamma_{mon}((-\infty,c_i]\times (-\infty,d_i]) = \min(F_\mu(c_i),F_\nu(d_i)) = F_\mu(c_i)\wedge F_\nu(d_i) 
\end{equation*}
The reason that we distinguish one-dimensional optimal transport is because in one dimension $c$-CM and monotonicity will correspond for a large class of cost functions, including the cost functions for the Wasserstein spaces.
In these cases, the monotone coupling will be the optimal coupling. 

\begin{theorem}[{\cite[Pg. 63]{santambrogio2015optimal}}]
\label{thm:1dot}
Let $h:\mathbb{R}\mapsto \mathbb{R}_+$ be a strictly convex function and $\mu,\nu\in\Pro(\mathbb{R})$ be probability measures. Consider the cost $c(x,y)=h(y-x)$ and suppose that the optimal transport cost is finite. Then, the optimal transport problem has a unique solution given by $\gamma_{mon}$, the monotone coupling  between $\mu$, and $\nu$.

Moreover, if strict convexity is withdrawn and h is only convex then the same $\gamma_{mon}$ is an optimal transport plan, but may no longer be unique.
\end{theorem}

Theorem \ref{thm:1dot} tells us that the monotone coupling is the optimal coupling for the appropriate cost functions.
We now prove a result that will be useful when we move to one-dimensional signed optimal transport.

\begin{lemma}
\label{lem:monotoneplanstructure}
Let $\mu$ and $\nu$ be measures with points $c_i$ and $d_i$ for $i$ from $1$ to $N-1$ such that $F_\mu(c_i)=F_\nu(d_i)\neq 0$. 
Let $\gamma_{mon}$ be the monotone coupling between them. 
Also let $c_0,d_0=-\infty$ and $c_{N}, d_{N}=\infty$.
Then $\supp(\gamma_{mon})\subset \cup_{i=1}^N (c_{i-1},c_i]\times (d_{i-1},d_i]$.
\end{lemma}
\begin{proof}
We prove this by showing that $\gamma_{mon}((c_i,\infty)\times(-\infty,d_i])=0$ and note that the proof would hold for any $i$ as well as for reversing the roles of $c_i$ and $d_i$.
Consider the following:
\begin{eqnarray*}
F_\nu(d_i)&=&\gamma_{mon}((-\infty,\infty)\times (-\infty,d_i])\\
&=& \gamma_{mon}((-\infty,c_i]\times(-\infty,d_i])+\gamma_{mon}((c_i,\infty)\times(-\infty,d_i])\\
&=& F_\mu(c_i)\wedge F_\nu(d_i)+\gamma_{mon}((c_i,\infty)\times(-\infty,d_i])\\
&=& F_\nu(d_i)+ \gamma_{mon}((c_i,\infty)\times(-\infty,d_i]).
\end{eqnarray*}
Subtracting $F_\nu(d_i)$ from both sides shows that $\gamma_{mon}((c_i,\infty)\times(-\infty,d_i])=0$.
This argument can be done in the same way starting with $F_\mu(c_i)$ to show that $\gamma_{mon}((-\infty,c_i]\times(d_i,\infty))=0$.
Together these prove our claim as we have shown that no mass lies in the region outside of 
$\cup_{i=1}^N (c_{i-1},c_i]\times (d_{i-1},d_i]$.

\end{proof}

\subsection{One-Dimensional Signed Optimal Transport}
We now look at measures on $\mathbb{R}$ which not only have finite mass, measure and moment, but also the property of finite-length signature, as defined by
\begin{definition}
A signed measure $a$ is said to have \textbf{finite-length signature} if it has finite mass and can be decomposed into $a=\sum_{i=1}^n a^i$ for some $n$, and the $a^i$ measures are mutually singular with support contained in an interval $(p_{i-1},p_i]$, with $p_0=-\infty$ and $p_n=\infty$ and each $a^i$ is equal to either $a^+ \restriction_ (p_{i-1},p_i]$ or $-a^-\restriction_ (p_{i-1},p_i]$. 
We further require that the measure $a^i$ has opposite sign to the measures $a^{i-1}$ and $a^{i+1}$ when they exist.
The signed measure $a$ then has \textbf{signature} $(z_1,...,z_n)$ with $z_i:=\int_{p_{i-1}}^{p_i} da$.
The intervals $(p_{i-1},p_i]$ are called the \textbf{signature intervals} of $a$.
\end{definition}

The signature of a measure is meant to encapsulate the order of the positive and negative mass of the measure. It has the following properties:
\begin{enumerate}
    \item $\int_\mathbb{R} da =\sum_j z_j$
    \item $\int_\mathbb{R} d\alpha =\sum_j \|z_j\|$
    \item $\int_{-\infty}^{p_i}da=\sum_{j=1}^{i} z_j$. 
\end{enumerate}

We now rephrase Lemma \ref{lem:monotoneplanstructure} in the context of signed measures with the same signature.
\begin{lemma}
\label{lem:signedmeasuremonotonestructure}
Let $\alpha$ and $\beta$ be positive measures coming from signed measures $a$ and $b$ that both have the same signature $(z_i)_{i=1}^n$ and with signature intervals $(p_{i-1},p_i]$ and $(q_{i-1},q_i]$. 
Then the support of the optimal plan $\gamma$ is contained in 
$$\cup_{i=1}^n  (p_{i-1},p_i]\times (q_{i-1},q_i].$$ 
\end{lemma}

With this lemma, we can show that any two signed measures with the same signature can be sent to one another by the kernel generated by the coupling between the absolute value measures. 

\begin{theorem}
\label{thm:signedcouplingskernel}
Let $a$ and $b$ be signed measures with the same finite-length signature $(z_1,...,z_n)$ and let $\alpha:=\|a\|$, and $\beta:=\|b\|$ be their absolute value measures with appropriate finite moments.
Let $\gamma$ be the optimal coupling between $\alpha$ and $\beta$, and $K^\gamma$ the kernel associated to $\gamma$.
Then $K^\gamma$ sends $a$ to $b$.
\end{theorem}
\begin{proof}
We will show that the optimal coupling $\gamma$ (which is also the monotone coupling) between $\alpha$ and $\beta$ is associated to a kernel that sends $a$ to $b$. 

Let $(p_{i-1},p_i]$ and $(q_{i-1},q_i]$ be the signature intervals for $a$ and $b$ respectively. 
From Theorem~\ref{lem:monotoneplanstructure} we know that $\supp(\gamma)\subset\cup_{i=1}^n (p_{i-1},p_i]\times(q_{i-1},q_i]$. 

Notice that the Radon-Nikodym derivative of $a^+$ with respect to $\alpha$ $f_+(x)$ is equal to $+1$ on $(p_{i-1},p_i]$ when $z_i$ is positive and $0$ otherwise, and the Radon-Nikodym derivative of $a^-$ with respect to $\alpha$ $f_-(x)$ is equal to  $+1$ on $(p_{i-1},p_i]$ when $z_i$ is negative and $0$ otherwise. 
Likewise, the Radon-Nikodym derivative of $b^+$ with respect to $\beta$ $g_+(x)$ is equal to $+1$ on $(q_{i-1},q_i]$ when $z_i$ is positive and $0$ otherwise, and the Radon-Nikodym derivative of $b^-$ with respect to $\beta$ $g_-(x)$ is equal to  $+1$ on $(q_{i-1},q_i]$ when $z_i$ is negative and $0$ otherwise.

Thus $f_+(x)\gamma(x,y) = g_+(y)\gamma(x,y)$ and $f_-(x)\gamma(x,y) = g_-(y)\gamma(x,y)$, as both $f_{\pm}(x)$ and $g_\pm(y)$ are constant over each block $(p_{i-1},p_{i}]\times (q_{i-1},q_i]$ with their values depending in the same manner on the sign of $z_i$. 

Letting $\gamma^+ = f_+(x)\gamma(x,y)$, and $\gamma^- = f_-(x)\gamma(x,y)$, it is immediate that they are $c$-CM compatible as they are components of $\gamma$ which is optimal between $\alpha$ and $\beta$. 
We also have $\int_Y d\gamma^+(x,y) = \int_Y f_+(x)d\gamma(x,y) =f_+(x)\alpha(x) = a^+$ and likewise for $b^+$. Similarly $\gamma^-$ will have marginals of $a^-$ and $b^-$. 

Thus from Theorem~\ref{thm:kernelsandsignedmeasures}, we know that $a$ and $b$ are connected by the optimal transport kernel associated to $\gamma$. 
\end{proof}

We now show that if two signed measures have the same integral and mass but different signatures, then they cannot be sent to each other. 

\begin{theorem}
\label{thm:differentsignatures}
Let $a$ and $b$ be signed measures with equal mass and integral and finite moment.
Let $(z_i)_{i=1}^{n}$ be the signature of $a$ and $(w_i)_{i=1}^m$ the signature of $b$, and let $(z_i)\neq(w_i)$. 
Then there is not an optimal transport kernel $K$ sending $a$ to $b$.
\end{theorem}
\begin{proof}
Suppose that there is an optimal transport kernel $K$ such that $aK = b$.
Since $a$ and $b$ have the same mass and integral, it must be the case that $a^+K=b^+$, and $a^-K = b^-$. 
Thus it must be the case that $\alpha K = \beta$. 

We then know that $\alpha\odot K = \gamma$ is the monotone plan between $\alpha$ and $\beta$. 
On measures that are absolutely continuous with respect to $\alpha$, we will have $\mu K = \mu K^\gamma$. 
This is the case for $a^\pm$. 

Thus it is the case that the kernel associated to the monotone plan that sends $a$ to $b$. 
But since we know that the signatures are different, there is mass from $a$ sent to mass of the opposite sign to $b$. This is a contradiction and completes the proof. 
\end{proof}

It is impossible to send two measures with equal mass  and integral, but which have different signatures, because the positive and negative components of two measures must be sent to each other.
When this is the case, the monotonicity requirement of the optimal couplings would force a mismatch as a kernel that is constrained to monotonicity cannot change the signature, (expect by destroying it). 
Throughout this section it has been important that the signed measures not only have equal integral, as would be expected from the image of a Markov kernel, but also equal mass. 
This is because an optimal transport kernel can destroy mass. 

Consider the kernel $K$ which maps $\chi_{[0,2]}$ to $2\delta_1$. 
Then the measure $a=\chi_{[0,1]}-\chi_{(1,2]}$ is a signed measure with finite mass, integral and moment, but $aK=0$. 
This is because the positive and negative parts of $a$ are sent to the same place and cancel out, but this is done without violating $c$-CM. 
Mass destruction is a feature one may want consider in future theoretical treatises  of signed optimal transport. 

\section{Conclusion}\label{sec:con}

This paper explains how Markov kernels naturally appear in optimal transport.
We showed not only that Markov kernels are natural objects to consider, but also that it is easy to use an optimal Markov kernel to obtain an optimal coupling and that we can consider an the kernel associated to an optimal coupling by restricting to measures absolutely continuous with respect to a marginal. 
By extending the theory to incorporate Markov kernels we broaden the theory and point to interesting connections like those between the Chapman-Kolmogorov equations and concatenating geodesics. 
Futher, the kernel perspective provided a way to extend optimal transport to signed measures. 
While there are restrictions that come from using Markov kernels for signed measures, they provide additional structure to the one-dimensional case. 

There is still necessary work to be done to extend the domain of kernels associated with an optimal coupling, but there is great opportunity to use this theory for new applications. 

While not touched upon here, it is a small step to see the potential in unbalanced optimal transport by opening up the theory to sub-Markov and super-Markov kernels. 
\section{Acknowledgements}
This work is partially supported by ONR \#N00014-20-1-2595. 
\bibliography{main}\bibliographystyle{plain}

\end{document}